\documentclass[12pt,reqno]{amsart}
\usepackage{amsmath,amssymb}
 \usepackage{cite}

 \makeatletter
 \evensidemargin  \oddsidemargin
 \makeatother

\newtheorem{theorem}{Theorem}
\theoremstyle{plain}

\newtheorem{proposition}{Proposition}

\numberwithin{equation}{section}

 \numberwithin{equation}{section}

\begin{document}
\title[Norm inequalities related to Heinz and Heron operator means]{Norm inequalities related to Heinz and Heron operator means}

\author[A.G. Ghazanfari]{A. G. Ghazanfari}

\address{Department of Mathematics, Lorestan University, P.O. Box 465, Khoramabad, Iran.}

\email{ghazanfari.a@lu.ac.ir}



\thanks{Data sharing not applicable to this article as no datasets were generated or analyzed during the current study.}

\begin{abstract}
In this article we present some new comparisons between the Heinz and Heron operator means,
which improve some recent results known from the literature. We derive some refinements of
these inequalities for unitarily invariant norms with the help of the contractive maps.
\end{abstract}
\maketitle

\section{Introduction and preliminaries}

Let $M_{m,n}(\mathbb{C})$ be the space of $m\times n$ complex matrices and $M_n(\mathbb{C})=M_{n,n}(\mathbb{C})$. Let $|||.|||$ denote any unitarily invariant
norm on $M_n(\mathbb{C})$. So,
$|||UAV|||=|||A|||$ for all $A\in M_n(\mathbb{C})$  and for all unitary matrices $U,V\in M_n(\mathbb{C})$. The Hilbert-Schmidt and trace class norm
of $A=[a_{ij}]\in M_n(\mathbb{C})$ are denoted by
\begin{equation*}
\|A\|_2=\left(\sum_{j=1}^n s_j^2(A)\right)^{\frac{1}{2}},\left\| A\right\|_1=\sum_{j=1}^n s_j(A)
\end{equation*}
 where $s_1(A)\geq s_2(A)\geq...\geq s_n(A)$ are the singular values of $A$, which are the eigenvalues of the positive semidefinite
 matrix $\mid A\mid =(A^*A)^{\frac{1}{2}}$, arranged in decreasing order and repeated according to multiplicity.
For Hermitian matrices $A,B\in M_n(\mathbb{C})$, we write that $A\geqslant  0$
if $A$ is positive semidefinite, $A>0$ if $A$ is positive definite, and $A\geqslant B$ if $A-B\geqslant 0$.

Let $A,B,X\in M_{n}$, $A,B$ are positive definite,
The matrix version of the Heinz means are defined by
\begin{equation*}
H_{\nu} (A,X,B)=\frac{A^\nu XB^{1-\nu}+A^{1-\nu} XB^\nu}{2},
\end{equation*}
and the matrix version of the Heron means are defined by
\begin{equation*}
\mathcal{F}_{\alpha}(A,X,B)=(1-\alpha)A^{\frac{1}{2}}XB^{\frac{1}{2}}+\alpha\left(\frac{AX+XB}{2}\right).
\end{equation*}

In 1998, Zhan \cite{zha} proved that if
$\frac{1}{4}\leq \nu \leq \frac{3}{4}$ and $\alpha \in [\frac{1}{2},\infty)$, then
 \begin{align}\label{1.1}
|||A^\nu XB^{1-\nu}+A^{1-\nu} XB^\nu|||\leq \frac{2}{2+t}|||AX+XB+tA^{\frac{1}{2}}XB^{\frac{1}{2}}|||,
\end{align}
for $-2<t\leq 2$.\\

Kaur and Singh \cite{kau} have shown that if
$\frac{1}{4}\leq \nu \leq \frac{3}{4}$ and $\alpha \in [\frac{1}{2},\infty)$, then
\begin{equation}\label{1.2}
|||H_{\nu} (A,X,B) |||\leq \left|\left|\left|\mathcal{F}_{\alpha}(A,X,B)\right|\right|\right|.
\end{equation}

Obviously, if $A,B,X\in M_{n}$, such that $A,B$ are positive definite, then for $\frac{1}{4}\leq \nu \leq \frac{3}{4}$
and $\alpha \in [\frac{1}{2},\infty)$, and any unitarily invariant norm $|||.|||$, the following inequalities hold
\begin{equation}\label{1.3}
|||A^{\frac{1}{2}}XB^{\frac{1}{2}}|||
\leq |||H_{\nu} (A,X,B)|||
\leq \left|\left|\left|\mathcal{F}_{\alpha}(A,X,B)\right|\right|\right|.
\end{equation}

I. Ali, H. Yang and A. shakoor \cite{ali} gave a refinement of the inequality \eqref{1.3} as follows:
\begin{align*}
|||H_{\nu} (A,X,B) |||\leq (4r_0-1)|||A^{\frac{1}{2}}XB^{\frac{1}{2}}|||
+2(1-2r_0)\left|\left|\left|\mathcal{F}_{\alpha}(A,X,B)\right|\right|\right|.
\end{align*}
where $\frac{1}{4}\leq \nu \leq \frac{3}{4}$, $\alpha \in [\frac{1}{2},\infty )$ and $r_{0}=\min\{\nu,1-\nu\}$.

Heretofore the inequalities discussed above are proved in the setting of  matrices.
Kapil and Singh in \cite{kap2}, using the contractive maps proved that the relation \eqref{1.2} holds for positive definite operators
on a complex Hilbert space $H$.
They also proved that if $A, B$ and $X$ are bounded linear operators on $H$ and $A, B$ are
positive definite operators and $\alpha \in[\frac{1}{2}, \infty)$, then
\begin{equation}\label{1.4}
\left|\left|\left|A^{\frac{1}{2}}XB^{\frac{1}{2}}\right|\right|\right|\leq\left|\left|\left|\int_0^1A^{\nu}XB^{1-\nu}d\nu\right|\right|\right|
\leq\left|\left|\left|\mathcal{F}_\alpha(A,X,B)\right|\right|\right|,
\end{equation}
the function $|||\mathcal{F}_{\alpha}|||$ is increasing for $\alpha \in [\frac{1}{2}, \infty)$ and
$|||\mathcal{F}_{\alpha}|||\leq |||\mathcal{F}_{\frac{1}{2}}|||$ for $\alpha\in [0, \frac{1}{2}]$.

Heinz and Heron inequalities involving unitarily invariant
norms have received renewed attention in recent years and a remarkable variety of refinements
and generalizations have been found (see, for example, \cite{kap1,kau2,kit,kit2}).

The aim of this paper is to obtain refinements or generalizations of the above mentioned inequalities in the setting of operators
for unitarily invariant norms.

\section{ Main results}

Let $B(H)$ denote the  set of all bounded linear operators on a complex
Hilbert space $H$. An operator $A\in B(H)$
is positive, and we write $A\geq 0$, if $(Ax,x)\geq0$ for every vector $x\in H$. If $A$ and $B$ are self-adjoint
operators, the order relation $A\geq B$ means, as usual, that $A-B$ is a positive operator.

To reach inequalities for bounded self-adjoint operators on Hilbert space, we shall use
the following monotonicity property for operator functions:\\
If $X\in B(H)$ is self adjoint with a spectrum $Sp(X)$, and $f,g$  are continuous real valued functions
on an interval containing $Sp(X)$, then
\begin{equation}\label{2.1}
f(t)\geq g(t),~t\in Sp(X)\Rightarrow ~f(X)\geq g(X).
\end{equation}
For more details about this property, the reader is referred to
\cite{pec}.

Let $L_X, R_Y$ denote the left and right multiplication maps on $B(H)$, respectively, that is,
$L_X(T)=XT$ and $R_Y(T)=TY$. Since $L_X$ and $R_Y$ commute, we have
\[
e^{L_X+R_Y}(T)=e^XTe^Y.
\]

Let $U$ be an invertible positive operator in $B(H)$, then there exists a self-adjoint operator $V \in B(H)$ such that $U=e^V$.
Let $n\in \mathbb{N}$ and $A, B$ be two invertible positive operators in $B(H)$. To simplify computations,
we denote $A$ and $B$ by $e^{2X_1}$ and $e^{2Y_1}$, respectively,
where $X_1$ and $Y_1$ in $B(H)$ are self-adjoint. The corresponding operator map $L_{X_{1}}-R_{Y_{1}}$ is denoted by $D$.
With these notations, we now use the results proved in \cite{kap2, lar} to derive the Heinz and Heron type inequalities
for unitarily invariant norms.

\begin{proposition}\label{p.1}
Let $0\leq s_2\leq s_1$ and $r,r'\leq\frac{s_1+s_2}{2}$ or $s_1\leq s_2\leq0$ and $r,r'\geq\frac{s_1+s_2}{2}$, then the following
operator maps are contractive,
\begin{enumerate}
\item[(1)]
$\frac{(1+t)\cosh (rD)}{\cosh(s_1D)+t\cosh(s_2D)}$, for $-1<t\leq1$.
\item[(2)]
$\frac{\alpha\cosh (rD)+(1-\alpha) \cosh(r'D)}{\beta\cosh(s_1D)+(1-\beta)\cosh(s_2D)}$, for $0\leq\alpha\leq1,~\frac{1}{2}\leq\beta$.
\item[(3)]
$\frac{(1+t)\sinh (rD)}{r(\sinh(s_1D)+t\sinh(s_2D))}$, for $-1<t\leq1$ and $|s_1+s_2|\geq2$.
\item[(4)]
$\frac{\frac{\alpha}{r}\sinh (rD)+\frac{1-\alpha}{r'} \sinh(r'D)}{\beta\sinh(s_1D)+(1-\beta)\sinh(s_2D)}$, for $0\leq\alpha\leq1,~\frac{1}{2}\leq\beta$ and $|s_1+s_2|\geq2$.

\end{enumerate}

\end{proposition}

\begin{proof}
\begin{enumerate}
\item[(1)] It is easy to see that
\begin{align*}
&\frac{(1+t)\cosh (rD)}{\cosh(s_1D)+t\cosh(s_2D)}\\
&=\frac{(1+t)\cosh (rD)}{\cosh(s_1D)+\cosh(s_2D)}\left(1-\frac{(1-t)\cosh (s_2D)}{\cosh(s_1D)+\cosh(s_2D)}\right)^{-1}\\
&=\frac{1+t}{2}\frac{2\cosh (rD)}{\cosh(s_1D)+\cosh(s_2D)}\times
\sum_{n=0}^\infty\left(\frac{1-t}{2}\right)^n\left(\frac{2\cosh (s_2D)}{\cosh(s_1D)+\cosh(s_2D)}\right)^n.
\end{align*}
For each $T\in B(H)$, using part (2) of Corollary 2.5 of \cite{kap2}, we obtain
\begin{align*}
\left|\left|\left|\frac{(1+t)\cosh (rD)}{\cosh(s_1D)+t\cosh(s_2D)}T\right|\right|\right|\leq
\frac{1+t}{2}\sum_{n=0}^\infty\left(\frac{1-t}{2}\right)^n|||T|||=|||T|||.
\end{align*}
\item[(2)] To prove part (2), we write
\begin{align*}
&\frac{\alpha\cosh (rD)+(1-\alpha) \cosh(r'D)}{\beta\cosh(s_1D)+(1-\beta)\cosh(s_2D)}\\
&=\alpha\frac{\cosh (rD)}{\beta\cosh(s_1D)+(1-\beta)\cosh(s_2D)}+(1-\alpha)\frac{\cosh (r'D)}{\beta\cosh(s_1D)+(1-\beta)\cosh(s_2D)}
\end{align*}
\end{enumerate}
Now, using part (1), we obtain
\begin{multline*}
\left|\left|\left|\frac{\cosh (rD)}{\beta\cosh(s_1D)+(1-\beta)\cosh(s_2D)}T\right|\right|\right|\\
=\left|\left|\left|\frac{\frac{1}{\beta}\cosh (rD)}{\cosh(s_1D)+(\frac{1}{\beta}-1)\cosh(s_2D)}T\right|\right|\right|\leq 1.
\end{multline*}
Similarly,
\begin{equation*}
\left|\left|\left|\frac{\cosh (r'D)}{\beta\cosh(s_1D)+(1-\beta)\cosh(s_2D)}T\right|\right|\right|\leq 1.
\end{equation*}
Consequently,
\begin{align*}
\left|\left|\left|\frac{\alpha\cosh (rD)+(1-\alpha) \cosh(r'D)}{\beta\cosh(s_1D)+(1-\beta)\cosh(s_2D)}T\right|\right|\right| \leq1.
\end{align*}
\item[(3)] A simple calculation shows that
\begin{align*}
&\frac{(1+t)\sinh (rD)}{r(\sinh(s_1D)+t\sinh(s_2D))}\\
&=\frac{1+t}{s_1+s_2}\frac{(s_1+s_2)\sinh (rD)}{r(\sinh(s_1D)+\sinh(s_2D))}\left(1-\frac{(1-t)\sinh (s_2D)}{\sinh(s_1D)+\sinh(s_2D)}\right)^{-1}\\
&=\frac{1+t}{s_1+s_2}\frac{(s_1+s_2)\sinh (rD)}{r(\sinh(s_1D)+\sinh(s_2D))}\\
&\times\sum_{n=0}^\infty\left(\frac{2s_2}{s_1+s_2}\right)^n\left(\frac{1-t}{2}\right)^n\left(\frac{(s_1+s_2)\sinh (s_2D)}{s_2(\sinh(s_1D)+\sinh(s_2D))}\right)^n.
\end{align*}
Using part (3) of Corollary 2.5 in \cite{kap2}, and $|\frac{2s_2}{s_1+s_2}|\leq1$, we deduce for each $T\in B(H)$
\begin{align*}
\left|\left|\left|\frac{(1+t)\sinh (rD)}{r(\sinh(s_1D)+t\sinh(s_2D))}T\right|\right|\right|&\leq
\frac{1+t}{|s_1+s_2|}\sum_{n=0}^\infty\left(\frac{1-t}{2}\right)^n|||T|||\\
&=\frac{2}{|s_1+s_2|}|||T|||\leq |||T|||.
\end{align*}
The proof of (4) is similar to (2) and we omitted it.
\end{proof}

\begin{theorem}\label{t.1}
Let $A,X,B\in B(H)$ and $A, B$ be two positive definite operators.
Let $\frac{1}{2}\leq\beta\leq 1$ and $\alpha \geq \frac{1}{2}$.
\begin{enumerate}
\item[(1)]
If $\frac{3}{8}\leq\nu\leq\frac{5}{8}$, then
\begin{align}\label{2.2}
|||H_{\nu} (A,X,B) |||
&\leq \left|\left|\left|(1-\beta)A^{\frac{1}{2}}XB^{\frac{1}{2}}+\beta H_{\frac{1}{4}} (A,X,B) \right|\right|\right|\notag\\
&\leq \left|\left|\left|\mathcal{F}_{\alpha}(A,X,B)\right|\right|\right|.
\end{align}

\item[(2)] If $\frac{5}{16}\leq\nu\leq\frac{11}{16}$, then
\begin{align}\label{2.3}
|||H_{\nu} (A,X,B) |||&\leq \left|\left|\left|(1-\beta)H_{\frac{3}{8}}(A,X,B)+\beta H_{\frac{1}{4}} (A,X,B)\right|\right|\right|\notag\\
&\leq \left|\left|\left|\mathcal{F}_\alpha(A,X,B)\right|\right|\right|.
\end{align}
\end{enumerate}
\end{theorem}

\begin{proof}

Put $T=A^\frac{1}{2}XB^\frac{1}{2}$, a calculus computation shows that
\begin{align*}
&A^\nu XB^{1-\nu}+A^{1-\nu}XB^\nu=A^{\nu-\frac{1}{2}}A^\frac{1}{2} XB^\frac{1}{2}B^{\frac{1}{2}-\nu}
+A^{\frac{1}{2}-\nu}A^\frac{1}{2}XB^\frac{1}{2}B^{\nu-\frac{1}{2}}\\
&=A^{\nu-\frac{1}{2}}TB^{\frac{1}{2}-\nu}
+A^{\frac{1}{2}-\nu}TB^{\nu-\frac{1}{2}}\\
&=e^{(2\nu-1)X_1}Te^{(1-2\nu)Y_1}+e^{(1-2\nu)X_1}Te^{(2\nu-1)Y_1}\\
&=e^{(2\nu-1)D}T+Te^{-(2\nu-1)D}=2\cosh((2\nu-1)D)T.
\end{align*}

Hence, we can write the operator on the left hand side of the first inequality in \eqref{2.2}  as
\begin{align}\label{2.4}
\frac{A^\nu XB^{1-\nu}+A^{1-\nu}XB^\nu}{2}=\cosh((2\nu-1)D)T
\end{align}
and the operator on the right hand side of the first inequality in \eqref{2.2} as
\begin{align*}
(1-\beta)A^{\frac{1}{2}}XB^{\frac{1}{2}}+\beta\frac{A^{\frac{1}{4}}XB^{\frac{3}{4}}
+A^{\frac{3}{4}}XB^{\frac{1}{4}}}{2}
=\left((1-\beta)+\beta\cosh \left(\frac{1}{2}D\right)\right)T.
\end{align*}
Considering
\begin{align*}
\frac{\cosh((2\nu-1)D)}{(1-\beta)+\beta\cosh(\frac{1}{2}D)}=\frac{\frac{1}{\beta}
\cosh((2\nu-1)D)}{\frac{(1-\beta)}{\beta}+\cosh(\frac{1}{2}D)}
\end{align*}
we put $t=\frac{1}{\beta}-1$, then
$|t|\leq1\Leftrightarrow \beta\geq\frac{1}{2}$.
Therefore, part (1) of Proposition \ref{p.1} proves the first inequality in \eqref{2.2}.
\begin{multline}\label{2.5}
\left|\left|\left|(1-\beta)A^{\frac{1}{2}}XB^{\frac{1}{2}}+\beta H_{\frac{1}{4}} (A,X,B) \right|\right|\right|\\
\leq (1-\beta)|||A^{\frac{1}{2}}XB^{\frac{1}{2}}|||+\beta |||H_{\frac{1}{4}} (A,X,B)|||\\
\leq (1-\beta)|||\mathcal{F}_\alpha(A,X,B)|||+\beta|||\mathcal{F}_\alpha(A,X,B)|||.
\end{multline}

Hence the proof of part (1) is completed.

Using the same arguments, we can show the operator on
the right hand side of the first inequality in \eqref{2.3} as
\begin{align}\label{2.6}
&(1-\beta)\frac{A^{\frac{3}{8}}XB^{\frac{5}{8}}+A^{\frac{5}{8}}XB^{\frac{3}{8}}}{2}
+\beta\frac{A^{\frac{1}{4}}XB^{\frac{3}{4}}+A^{\frac{3}{4}}XB^{\frac{1}{4}}}{2}\notag\\
&=\left((1-\beta)\cosh\left(\frac{1}{4}D\right)+\beta\cosh\left(\frac{1}{2}D\right)\right)T.
\end{align}

By part (1) of Proposition \ref{p.1} the function
$\frac{\cosh((2\nu-1)D)}{(1-\beta)\cosh\left(\frac{1}{4}D\right)+\beta\cosh\left(\frac{1}{2}D\right)}$
is contractive, therefore the first inequality in \eqref{2.3} holds. The Similar to relation \eqref{2.5},
the second inequality in \eqref{2.3} is proved.
Hence the proof is completed.

\end{proof}

\begin{theorem}\label{t.2}
Let $A,X,B\in B(H)$ and $A, B$ be two positive definite operators. Let
$\frac{1}{2}\leq \beta, \gamma, \delta\leq1$ and $\alpha\geq\frac{1}{2}$.

\begin{enumerate}

\item[(1)] If $\frac{9}{32}\leq\nu\leq\frac{23}{32}$, then
\begin{align}\label{2.7}
|||H_{\nu} (A,X,B) |||
&\leq \left|\left|\left|(1-\beta)H_{\frac{5}{16}} (A,X,B)+\beta H_{\frac{1}{4}} (A,X,B)
\right|\right|\right|\notag\\
&\leq \left|\left|\left|\mathcal{F}_\alpha(A,X,B)\right|\right|\right|.
\end{align}

\item[(2)] If $\frac{11}{32}\leq\nu\leq\frac{21}{32}$, then
\begin{align}\label{2.8}
|||H_{\nu} (A,X,B) |||
&\leq \left|\left|\left|(1-\gamma)H_{\frac{3}{8}} (A,X,B)+\gamma H_{\frac{5}{16}} (A,X,B)
\right|\right|\right|\notag\\
&\leq \left|\left|\left|(1-\beta)H_{\frac{3}{8}} (A,X,B)+\beta H_{\frac{1}{4}} (A,X,B)
\right|\right|\right|\notag\\
&\leq \left|\left|\left|\mathcal{F}_\alpha(A,X,B)\right|\right|\right|.
\end{align}
\end{enumerate}

\end{theorem}

\begin{proof}
By the same arguments and notations used earlier,
we can show the operator on the right hand side of the first inequality in \eqref{2.7} as

\begin{align*}
&(1-\beta)\frac{A^{\frac{5}{16}}XB^{\frac{11}{16}}+A^{\frac{11}{16}}XB^{\frac{5}{16}}}{2}
+\beta\frac{A^{\frac{1}{4}}XB^{\frac{3}{4}}+A^{\frac{3}{4}}XB^{\frac{1}{4}}}{2}\notag\\
&=\left((1-\beta)\cosh\left(\frac{3}{8}D\right)+\beta\cosh\left(\frac{1}{2}D\right)\right)T.
\end{align*}
Being contractive the map $\frac{\cosh((2\nu-1)D)}{(1-\beta)\cosh\left(\frac{3}{8}D\right)+\beta\cosh\left(\frac{1}{2}D\right)}$
 proves the first inequality in \eqref{2.7}. The Similar to relation \eqref{2.5},
the second inequality in \eqref{2.7} is proved.

It is easy to see that
\begin{align*}
&(1-\gamma)\frac{A^{\frac{3}{8}}XB^{\frac{5}{8}}+A^{\frac{5}{8}}XB^{\frac{3}{8}}}{2}
+\gamma\frac{A^{\frac{5}{16}}XB^{\frac{11}{16}}+A^{\frac{11}{16}}XB^{\frac{5}{16}}}{2}\notag\\
&=\left((1-\gamma)\cosh\left(\frac{1}{4}D\right)+\gamma\cosh\left(\frac{3}{8}D\right)\right)T.
\end{align*}
The maps $\frac{\cosh((2\nu-1)D)}{(1-\gamma)\cosh\left(\frac{1}{4}D\right)+\gamma\cosh\left(\frac{3}{8}D\right)}$
and $\frac{(1-\gamma)\cosh\left(\frac{1}{4}D\right)\gamma\cosh\left(\frac{3}{8}D\right)}{(1-\beta)\cosh\left(\frac{1}{4}D\right)\beta\cosh\left(\frac{1}{2}D\right)}$
are contractive. Therefore we get the desired results in  part (2).
\end{proof}

In the following Theorem, we give a refinement of \eqref{1.4}.

\begin{theorem}\label{t.3}
Let $A,X,B\in B(H)$ and $A, B$ be two positive definite operators. If $\frac{1}{4}<\nu<\frac{3}{4}$ and $\alpha\geq \frac{1}{2}$, then
\begin{align}\label{2.9}
|||H_\nu(A,X,B)|||&\leq\left|\left|\left|\frac{1}{2}H_\frac{1}{8}(A,X,B)+\frac{1}{2}H_\frac{3}{8}(A,X,B)\right|\right|\right|
\leq\left|\left|\left|\int_0^1A^{t}XB^{1-t}dt\right|\right|\right|\notag\\
&\leq \left|\left|\left|\frac{1}{2}H_\frac{1}{4}(A,X,B)+\frac{1}{2}F_\frac{1}{2}(A,X,B)\right|\right|\right|\notag\\
&\leq\left|\left|\left|\mathcal{F}_\alpha(A,X,B)\right|\right|\right|.
\end{align}

\end{theorem}

\begin{proof}
By the relation \eqref{2.3}, we have
\begin{align*}
H_\nu(A,X,B)=\cosh((2\nu-1)D)T.
\end{align*}

Again as in previous theorems, we get
\begin{align*}
\frac{1}{2}H_\frac{1}{8}(A,X,B)+\frac{1}{2}H_\frac{3}{8}(A,X,B)&=\left(\frac{1}{2}\cosh\left(\frac{3}{4}D\right)
+\frac{1}{2}\cosh\left(\frac{1}{4}D\right)\right)T,
\end{align*}

\begin{align}\label{2.90}
\int_0^1A^{\nu}XB^{1-\nu}d\nu=\int_0^1\exp((2\nu-1)D)T d\nu=(D^{-1}\sinh D )T,
\end{align}
and
\begin{align*}
\frac{1}{2}H_\frac{1}{4}(A,X,B)+\frac{1}{2}F_\frac{1}{2}(A,X,B)=\left(\frac{1}{2}\cosh\left(\frac{1}{2}D\right)
+\frac{1}{2}\left(\frac{1+\cosh(D)}{2}\right)\right)T.
\end{align*}
We get the first inequality in \eqref{2.9}, since the map $\frac{2\cosh(2\nu-1)D}{\cosh(\frac{3}{4}D)+\cosh(\frac{1}{4}D)}$ is contractive.

By part (3) of the Proposition 2.3 in \cite{kap2}, the map $\frac{D(\cosh(\frac{3}{4}D)+\cosh(\frac{1}{4}D))}{2\sinh D}=\frac{D\cosh(\frac{1}{4}D)}{2\sinh D}$ is contractive.
This proves the second inequality in \eqref{2.9}.

By part (4) of the Proposition 2.3 in \cite{kap2}, the following map is contractive.
\begin{align*}
&\frac{2\sinh D}{D\left(\cosh\left(\frac{1}{2}D\right)+\frac{1+\cosh(D)}{2}\right)}
=\frac{2\sinh D}{D\left(\cosh\left(\frac{1}{2}D\right)+\cosh^2(\frac{1}{2}D)\right)}\\
&=\frac{2\sinh(\frac{1}{2}D)\cosh(\frac{1}{2}D)}{D\left(\cosh\left(\frac{1}{2}D\right)
\cosh^2(\frac{1}{4}D)\right)}=\frac{4\sinh(\frac{1}{4}D)}{D\cosh\left(\frac{1}{4}D\right)}.
\end{align*}
This proves the third inequality in \eqref{2.9}.

\begin{align*}
\left|\left|\left|\frac{1}{2}H_\frac{1}{4}(A,X,B)+\frac{1}{2}F_\frac{1}{2}(A,X,B)\right|\right|\right|
\leq  \frac{1}{2}|||H_\frac{1}{4}(A,X,B)|||+\frac{1}{2}|||F_\frac{1}{2}(A,X,B)|||\\
\leq \frac{1}{2}|||F_\alpha(A,X,B)|||+\frac{1}{2}|||F_\alpha(A,X,B)|||=|||F_\alpha(A,X,B)|||.
\end{align*}
This completes the proof.
\end{proof}

As an application, we give the following inequalities for unitarily invariant norms.
From inequality \eqref{1.2}, we get
\begin{align*}
\left|\left|\left| \int_{\frac{1}{4}}^{\frac{3}{4}}H_\nu(A,X,B)d\nu\right|\right|\right|&\leq\int_{\frac{1}{4}}^{\frac{3}{4}}|||H_\nu(A,X,B)|||d\nu\\
&\leq\int_{\frac{1}{4}}^{\frac{3}{4}}|||F_\alpha(A,X,B)|||d\nu=\frac{1}{2}|||F_\alpha(A,X,B)|||.
\end{align*}

Hence from inequalities \eqref{2.2}, \eqref{2.7} and \eqref{2.8} respectively, we get
\begin{align*}
&\left|\left|\left| \int_{\frac{3}{8}}^{\frac{5}{8}}H_\nu(A,X,B)d\nu\right|\right|\right|\leq\frac{1}{4}|||F_\alpha(A,X,B)|||,\\
&\left|\left|\left| \int_{\frac{9}{32}}^{\frac{23}{32}}H_\nu(A,X,B)d\nu\right|\right|\right|\leq\frac{7}{16}|||F_\alpha(A,X,B)|||
\end{align*}
and
\begin{align*}
\left|\left|\left| \int_{\frac{11}{32}}^{\frac{21}{32}}H_\nu(A,X,B)d\nu\right|\right|\right|\leq\frac{5}{16}|||F_\alpha(A,X,B)|||.
\end{align*}

As other applications, we now consider some inequalities that involve differences and sums of operators.
These inequalities in the operator setting can be seen in \cite{kap2}, and in the matrix setting can be seen in \cite{sin}.
In the following theorems, we give new forms of the presented inequalities in Theorem 4.4 and Theorem 4.10 in \cite{sin},
and inequality (3.25) in \cite{kap2}. We also give complementary inequalities to them.

Let $p\in \mathbb{R}$, $A,X,B\in B(H)$ and $A, B$ be two positive definite operators.
For $p\in \mathbb{R}$, we introduce the function $f(\nu)=|||A^\nu XB^{p-\nu}+A^{p-\nu}XB^\nu|||$.
It is easy to show that the function $f(\nu)$ is decreasing for $\nu\leq \frac{p}{2}$, and increasing for $\nu\geq \frac{p}{2}$.
The part (2) of Proposition \ref{p.1}, implies that $f(\nu)$ is convex on $\mathbb{R}$.

\begin{theorem}\label{t.5}
Let $r,p,\nu\in \mathbb{R}$, $~A,X,B\in B(H)$ and $A, B$ be two positive definite operators.
If $\nu\leq 2r\leq p$ or $\nu\geq 2r\geq p$, then for $|t|\leq 1$

\begin{align}\label{2.10}
(1)~(1+t)&|||A^r XB^{p-r}+A^{p-r}XB^r|||\notag\\
&\leq |||A^pX+t(A^\nu XB^{p-\nu}+A^{p-\nu}XB^{\nu})+XB^p)|||.
\end{align}
Furthermore, if $|p-\nu|\geq 1$, then
\begin{align}\label{2.11}
(2&)~(1+t)|||A^r XB^{p-r}-A^{p-r}XB^r|||\notag\\
&\leq |p-2r|~|||A^pX+t(A^\nu XB^{p-\nu}-A^{p-\nu}XB^{\nu})-XB^p)|||.
\end{align}

\end{theorem}

\begin{proof}
\begin{enumerate}
\item[(1)]
Notice that
\begin{align*}
A^\nu XB^{p-\nu}+A^{p-\nu}XB^\nu=A^{\nu-\frac{p}{2}}A^{\frac{p}{2}}XB^{\frac{p}{2}}B^{\frac{p}{2}-\nu}
+A^{\frac{p}{2}-\nu}A^{\frac{p}{2}}XB^{\frac{p}{2}}B^{\nu-\frac{p}{2}}.
\end{align*}
Put $U=A^{\frac{p}{2}}XB^{\frac{p}{2}}$. Then
\begin{align*}
A^\nu XB^{p-\nu}+A^{p-\nu}XB^\nu&=A^{\nu-\frac{p}{2}}UB^{\frac{p}{2}-\nu}
+A^{\frac{p}{2}-\nu}UB^{\nu-\frac{p}{2}}\\
&=2(\cosh(p-2\nu)D)U,
\end{align*}

\begin{align*}
A^p X+XB^p=A^{\frac{p}{2}}UB^{-\frac{p}{2}}
+A^{-\frac{p}{2}}UB^{\frac{p}{2}}=2(\cosh pD)U.
\end{align*}
Using part (1) of Proposition \ref{p.1}, we get the following map is contractive.
\begin{align*}
\frac{(1+t)\cosh((p-2r)D)}{\cosh(pD)+t\cosh(p-2\nu)D}.
\end{align*}
\item[(2)]
We know that
\begin{align*}
A^\nu XB^{p-\nu}-A^{p-\nu}XB^\nu=A^{\nu-\frac{p}{2}}A^{\frac{p}{2}}XB^{\frac{p}{2}}B^{\frac{p}{2}-\nu}
-A^{\frac{p}{2}-\nu}A^{\frac{p}{2}}XB^{\frac{p}{2}}B^{\nu-\frac{p}{2}}.
\end{align*}
Put $U=A^{\frac{p}{2}}XB^{\frac{p}{2}}$. Then
\begin{align*}
A^\nu XB^{p-\nu}-A^{p-\nu}XB^\nu&=A^{\nu-\frac{p}{2}}UB^{\frac{p}{2}-\nu}
-A^{\frac{p}{2}-\nu}UB^{\nu-\frac{p}{2}}\\
&=2(\sinh(p-2\nu)D)U,
\end{align*}

\begin{align*}
A^p X-XB^p=A^{\frac{p}{2}}UB^{-\frac{p}{2}}
-A^{-\frac{p}{2}}UB^{\frac{p}{2}}=2(\sinh pD)U.
\end{align*}
Using part (3) of Proposition \ref{p.1}, we get the following map is contractive.
\begin{align*}
\frac{(1+t)\sinh((p-2r)D)}{(p-2r)(\sinh(pD)+t\sinh(p-2\nu)D)}.
\end{align*}
\end{enumerate}
\end{proof}

In special case, the relation \eqref{2.10} for $p=1$ and $\nu=\frac{1}{2}$, becomes the following inequality
\begin{align*}
|||A^r X&B^{1-r}+A^{1-r}XB^r|||\leq\frac{2}{2+t}\left|\left|\left|tA^\frac{1}{2}XB^\frac{1}{2}+\frac{AX+XB}{2}\right|\right|\right|,
\end{align*}
which holds for $\frac{1}{4}\leq r \leq \frac{3}{4}$ and $-2<t\leq 2$. In this case the relation \eqref{1.1} follows from \eqref{2.10}.

\begin{theorem}\label{t.6}
Let $r,p,\nu\in \mathbb{R}$, $~A,X,B\in B(H)$ and $A, B$ be two positive definite operators.
If $r\leq 0\leq p$ and $r\leq \nu\leq\frac{p}{2}$ (or, $r\geq 0\geq p$ and $r\geq \nu\geq\frac{p}{2}$), then for $t\geq0$,

\begin{align}\label{2.12}
(1)~(1+t)~|||A^r X&B^{p-r}+A^{p-r}XB^r|||\notag\\
&\geq |||A^pX+t(A^\nu XB^{p-\nu}+A^{p-\nu}XB^{\nu})+XB^p)|||,
\end{align}
and
\begin{align}\label{2.13}
(2)~\left(\frac{(1+t)p-2t\nu}{p-2r}\right)& |||A^r XB^{p-r}-A^{p-r}XB^r|||\notag\\
&\geq  |||A^pX+t(A^\nu XB^{p-\nu}-A^{p-\nu}XB^{\nu})-XB^p)|||.
\end{align}

\end{theorem}

\begin{proof}
First, we prove inequality \eqref{2.13}. By part (4) of the Proposition 2.4 in \cite{kap2}, we get
\begin{align*}
|||A^pX-XB^p|||\leq \left(\frac{p}{p-2r}\right) |||A^r XB^{p-r}-A^{p-r}XB^r|||,
\end{align*}
and
\begin{align*}
|||A^\nu XB^{p-\nu}-A^{p-\nu}XB^{\nu}|||\leq \left(\frac{p-2\nu}{p-2r}\right) |||A^r XB^{p-r}-A^{p-r}XB^r|||.
\end{align*}
From the following triangle inequality
\begin{align*}
|||A^pX+t(A^\nu XB^{p-\nu}&-A^{p-\nu}XB^{\nu})-XB^p)|||\\
&\leq |||A^pX-XB^p|||+t~|||A^\nu XB^{p-\nu}-A^{p-\nu}XB^{\nu}|||,
\end{align*}
we deduce the inequality \eqref{2.13}. Similarly, the inequality \eqref{2.12} is proved.
\end{proof}

\bibliographystyle{amsplain}

\end{document}